\documentclass[11pt,reqno]{amsart}

\usepackage{amssymb}
\usepackage{amsmath}
\usepackage{amscd}
\usepackage{amstext}
\usepackage{amsfonts}
\usepackage[all]{xy}

\newtheorem{thm}{Theorem}[section] 

\newtheorem{prop}[thm]{Proposition}
\newtheorem{cor}[thm]{Corollary}
\newtheorem{rem}[thm]{Remark}
\theoremstyle{definition}
\newtheorem{definition}[thm]{Definition}

\newcommand{\C}{\mathbb{C}}

\numberwithin{equation}{section}

\begin{document} 

\title[Representations of  the fundamental groups of Sasakian manifolds]{Almost-formality and deformations of representations of  the fundamental groups of Sasakian manifolds}

\author[H. Kasuya]{Hisashi Kasuya}

\address{Department of Mathematics, Graduate School of Science, Osaka University, Osaka,
Japan}

\email{kasuya@math.sci.osaka-u.ac.jp}



\begin{abstract}
For a $2n+1$-dimensional compact Sasakian manifold, if $n\ge 2$, we prove that the analytic germ of the variety  of representations of the fundamental group at every semi-simple representation is quadratic.
To prove this result, we prove the almost-formality of de Rham complex of a Sasakian manifold with values in a semi-simple flat vector bundle.
By  the almost-formality, we also prove the vanishing theorem on the cup product of the cohomology of semi-simple flat vector bundles over a   compact Sasakian manifold.
\end{abstract}

\maketitle
\section{Introduction}
In \cite{GM}, Goldman and Millson establishes the correspondence between local structures of the varieties of representations of the fundamental groups of manifolds and the deformation theory of augmented differential graded  Lie algebras.
By using this correspondence, they proves that  for a compact K\"ahler manifold $M$ with a point $x$,   the analytic germ of the variety ${\mathcal R}(\pi_{1}(M,x), GL_{m}(\C))$ of representations at $\rho$  is quadratic for any representation $\rho :\pi_{1}(M,x)\to GL_{m}(\C)$ which is the monodromy representation of a polarized variation of Hodge structure.
In \cite{SimL}, Simpson generalizes this result  for  every semi-simple representation $\rho :\pi_{1}(M,x)\to GL_{m}(\C)$ by using harmonic bundle structures of semi-simple flat bundles over compact  K\"ahler manifolds.

The aim of this paper is to give an  odd-dimensional analogue of this result.
A Sasakian manifold is  viewed as an important odd-dimensional analogue of a K\"ahler manifold (\cite{Sa}).
The main result of this paper is to prove the following statement.
\begin{thm}\label{MTS}
Let $M$ be a $(2n+1)$-dimensional compact Sasakian manifold with a point $x$.
If $n\ge 2$, then for any semi-simple representation $\rho :\pi_{1}(M,x)\to GL_{m}(\C)$ of the fundamental group, the analytic germ of the variety ${\mathcal R}(\pi_{1}(M,x), GL_{m}(\C))$ of representations at $\rho$  is quadratic.
\end{thm} 

To  prove  the quadraticity,  it is important   to reduce the  de Rham complex of  a compact K\"ahler manifold  with values in a flat bundle to its cohomology   via quasi-isomorphisms so-called  the  formality as in \cite{DGMS}.
In  \cite{SimL}, by using harmonic bundle structures, Simpson proves the formality for every semi-simple flat bundle over a compact K\"ahler manifold.
To prove Theorem \ref{MTS}, we prove the "almost-formality" of  the de Rham complex of a Sasakian manifold  with values in a semi-simple  flat bundle (Corollary \ref{alfo}) by using (basic) harmonic bundle structures as in \cite{BK}.

We have another application of such almost-formality.
By the almost-formality, we also prove the following statement in the cup product.
\begin{thm}
Let $M$ be a $(2n+1)$-dimensional compact Sasakian manifold and $E$ and $E^{\prime}$ be semi-simple  flat complex vector bundles.
For $s, t<n$ with $s+t> n$,  the cup product
\[H^{s}(M, \,E)\otimes H^{t}(M, \,E^{\prime})\to H^{s+t}(M, \,E\otimes E^{\prime})
\]
vanishes.
\end{thm}

\section{Harmonic bundles}

Let $M$ be a compact Riemannian manifold and $E$ a flat complex vector bundle over $M$
equipped with a flat connection $D$.
For any Hermitian metric $h$ on $E$, we have a unique decomposition
\begin{equation}\label{code}
D=\nabla+\phi
\end{equation}
such that $\nabla$ is a unitary connection and $\phi$ is a $1$-form on $M$ with values in
the self-adjoint endomorphisms of $E$ with respect to $h$.

\begin{thm}[\cite{Cor}]
A flat complex vector bundle $(E,\, D)$ is semi-simple if and only if  there exists a Hermitian metric
(called the harmonic metric) $h$ on $E$ such that
\[
\nabla^{\ast}\phi=0,
\]
where $\nabla^{\ast}$ is the formal adjoint operator of $\nabla$.

\end{thm}

\section{Sasakian manifolds}

Let $M$ be a $(2n+1)$-dimensional real smooth manifold. A {\em CR-structure} on $M$ is an 
$n$-dimensional complex sub-bundle $T^{1,0}$ of the complexified tangent bundle $TM_{\C}\,=\, 
TM\otimes_{\mathbb R} {\C}$ such that $T^{1,0}\cap \overline{T^{1,0}}=\{0\}$ and $T^{1,0}$ is 
integrable. We 
shall denote $\overline{T^{1,0}}$ by $T^{0,1}$. For a CR-structure $T^{1,0}$ on $M$, there is 
a unique sub-bundle $S$ of rank $2n$ of the real tangent bundle $TM$ together with a
vector bundle homomorphism $I\,:\,S\,\longrightarrow\, S$ satisfying the conditions that
\begin{enumerate}
\item $I^{2}=-{\rm Id}_{S}$, and

\item $T^{1,0}$ is the $\sqrt{-1}$-eigenbundle of $I$.
\end{enumerate}

A $(2n+1)$-dimensional manifold $M$ equipped with a triple $(T^{1,0},\, S,\, I)$ as above is 
called a {\em CR-manifold}. A {\em contact CR-manifold} is a CR-manifold $M$ with a contact 
$1$-form $\eta$ on $M$ such that $\ker\eta=S$. Let $\xi$ denote the Reeb vector field for the 
contact form $\eta$. On a contact CR-manifold, the above homomorphism $I$ extends to entire 
$TM$ by setting $I(\xi)\,=\,0$.

\begin{definition}
A contact CR-manifold $(M,\, (T^{1,0},\, S,\, I),\, (\eta,\, \xi))$ is a {\em 
strongly pseudo-convex CR-manifold} if the Hermitian form $L_{\eta}$ on $S_x$ defined by 
$L_{\eta}(X,Y)=d\eta(X, IY)$, $X,Y\,\in S_{x}$, is positive definite for every point
$x\,\in\, M$. 
\end{definition}

For a strongly pseudo-convex CR-manifold $(M, \,(T^{1,0},\, S, \,I),\, (\eta,\, \xi))$, we have
 a canonical Riemann metric $g_{\eta}$ on $M$ which is defined by
\[g_{\eta}(X,Y):= L_{\eta}(X,Y)+\eta(X)\eta(Y)\, ,\ \ X,\,Y\,\in \,T_{x}M\ .
\]

\begin{definition}
A Sasakian manifold is a strongly pseudo-convex CR-manifold $$(M, \,(T^{1,0},\, S,\, I),\,
(\eta,\, \xi))$$ such that for any section $\zeta$ of $T^{1,0}$, $[\xi, \zeta]$ is also a section of $T^{1,0}$.

\end{definition}
For a  Sasakian manifold $(M, \,(T^{1,0},\, S,\, I),\,
(\eta,\, \xi))$, the metric cone of   $(M, g_{\eta})$ is a K\"ahler manifold.
We can also define a Sasakian manifold as a contact metric manifold whose metric cone is   K\"ahler (see \cite{BoG}).

\section{Almost-formality}

Let $(M,\, (T^{1,0}, \,S,\, I),\, (\eta, \,\xi))$ be a compact Sasakian manifold. Then the 
Reeb vector field $\xi$ defines a $1$-dimensional foliation $\mathcal F_{\xi}$ on $M$. It is 
known that the map $I:TM\to TM$ associated with the CR-structure $T^{1,0}$ 
defines a transversely complex structure on the foliated manifold $(M,\,\mathcal F_{\xi})$. 
Furthermore, 
the closed basic $2$-form $d\eta$ is a transversely K\"ahler structure with respect
to this transversely complex structure. 

A  differential form $\omega$ on $M$ is called {\em basic} if  the equations
\begin{equation}\label{bafo}
i_{\xi}\omega=0= {\mathcal L}_{\xi}\omega
\end{equation}
hold.
We denote by $A^{\ast}_{B}(M)$ the subspace of basic 
forms in the de Rham complex $A^{\ast}(M)$.
 Then
$A^{\ast}_{B}(M)$ is a sub-complex of the de Rham complex $A^{\ast}(M)$.

Corresponding to the
decomposition $S_{\C}=T^{1,0}\oplus T^{0,1}$, we have the bigrading $$A^{r}_{B}(M)_{\C}=\bigoplus_{p+q=r} A^{p,q}(M)$$ as well as the decomposition of the exterior 
differential $$d_{\vert A^{r}_{B}(M)_{\C}}=\partial_{\xi}+\overline\partial_{\xi}$$ on $A^{r}_{B}(M)_{\C}$, so that $$\partial_{\xi}\,:\,A^{p,q}_{B}(M)\to
A^{p+1,q}_{B}(M)\ \text{ and } \
\overline\partial_{\xi}\,:\,A^{p,q}_{B}(M)\to 
A^{p,q+1}_{B}(M)\, .$$
 
We have the transverse Hodge theory as in (\cite{KT}, \cite{EKA}).
Consider the usual Hodge star operator $$\ast: A^{r}(M)\to A^{2n+1-r}(M)$$
associated to the Sasakian metric $g_{\eta}$ and the formal adjoint operator
$$\delta=-\ast d\ast :A^{r}(M)\to  A^{r-1}(M) .$$
We define the homomorphism $$\star_{\xi}:A^{r}_{B}(M)
\to A^{2n-r}_{B}(M)$$ to be
$\star_{\xi}\omega=\ast(\eta\wedge \omega)$ for $\omega\in A^{r}_{B}(M)$.
Also define the operators $$\delta_{\xi}=-\star_{\xi}d\star_{\xi}:
A^{r}_{B}(M)\to A^{r-1}_{B}(M),$$
$$\partial_{\xi}^{\ast}\,=-\star_{\xi}\overline\partial_{\xi}\star_{\xi}:
A^{p,q}_{B}(M)\to A^{p-1,q}_{B}(M)\, ,$$
$$\overline\partial_{\xi}^{\ast}\,=\,-\star_{\xi}\partial_{\xi}\star_{\xi}\,:\,
A^{p,q}_{B}(M)\,\longrightarrow\, A^{p,q-1}_{B}(M)$$ and
$\Lambda \,=\,-\star_{\xi}(d\eta\wedge)\star_{\xi}$.
They are the formal adjoints of $d$, $\partial_{\xi}$, $\overline\partial_{\xi}$ and
$(d\eta\wedge)$ respectively for the pairing 
\[A^{r}_{B}(M)\times A^{r}_{B}(M)\,\ni\,
(\alpha,\,\beta)\,\longmapsto\, \int_{M} \eta\wedge\alpha\wedge \star_{\xi}\beta\, .
\]
Define the Laplacian operators $$\Delta\,:\, A^{r}(M)\,\longrightarrow\, A^{r}(M)\ \ \text{ and }
\ \ \Delta_{\xi}\,:\,A^{r}_{B}(M)\,\longrightarrow\,
A^{r}_{B}(M)$$ by
\[\Delta\,=\,d\delta+\delta d\ \ \text{ and } \ \ \Delta_{\xi}\,=\,d\delta_{\xi}+\delta_{\xi}d
\]
respectively. For $\omega\,\in\, A^{r}_{B}(M)$, since the relation $\ast\omega
\,=\,(\star_{\xi}\omega)\wedge \eta$ holds, we have the relation 
\[\delta\omega\,=\,\delta_{\xi}\omega+\ast (d\eta\wedge \star_{\xi}\omega)\, .
\]
Thus, for $\omega\,\in\, A^{1}_{B}(M)$, the equality $\delta_{\xi}\omega
\,=\,\delta\omega$ holds, and hence for $f\,\in\, A^{0}_{B}(M)$, we have
that $\Delta_{\xi}f\,=\,\Delta f$.
The usual K\"ahler identities
\[[\Lambda, \partial_{\xi}]\,=\,
-\sqrt{-1}\overline\partial_{\xi}^{\ast}\ \ \text{ and } \ \
[\Lambda ,\overline\partial_{\xi}]=\sqrt{-1}\partial_{\xi}^{\ast}
\]
hold, and these imply that
\[\Delta_{\xi}\,=\,2\Delta_{\xi}^{\prime}\,=\,2\Delta_{\xi}^{\prime\prime}\, ,
\]
where $\Delta_{\xi}^{\prime}\,=\,\partial_{\xi}\partial_{\xi}^{\ast}+\partial_{\xi}^{\ast}\partial_{\xi}$
and $\Delta_{\xi}^{\prime\prime}\,=\,\overline\partial_{\xi}\overline\partial_{\xi}^{\ast}+
\overline\partial_{\xi}^{\ast}\overline\partial_{\xi}$.

A {\em basic vector bundle} $E$ over the foliated manifold  $(M,\,\mathcal F_{\xi})$ is a $C^\infty$ vector bundle over $M$ which  has local trivializations with respect to an open covering 
$M\,=\,\bigcup_{\alpha} U_{\alpha}$ satisfying the condition that each transition function 
$f_{\alpha\beta}:U_{\alpha}\cap U_{\beta}\to {\rm GL}_{r}(\C)$ is basic 
on $U_{\alpha}\cap U_{\beta}$ i.e.  it is constant on the leaves of the foliation 
$\mathcal F_{\xi}$.
For a basic vector bundle $E$, a differential form $\omega\,\in\, A^{\ast}(M,\,E)$ with 
values in $E$ is called basic if $\omega$ is basic on every $U_{\alpha}$, meaning 
$\omega_{\vert U_{\alpha}}\,\in\, A^{\ast}_{B_{\mathcal F}}(U_{\alpha})\otimes \C^{r}$ for 
every $\alpha$. Let $$A^{\ast}_{B}(M,\, E)\,\subset\,A^{\ast}(M,\,E)$$ denote the 
subspace of basic forms in the space $A^{\ast}(M,\,E)$ of differential forms with values in 
$E$.
Corresponding to the
decomposition $S_{\C}\,=\,T^{1,0}\oplus T^{0,1}$, we have the bigrading $$A^{r}_{B}(M,\,E)\,=\,\bigoplus_{p+q=r} A^{p,q}(M,\, E).$$

We shall consider any flat vector bundle $(E,\, D)$ over $M$ as a basic vector 
bundle by local flat frames. Then, $A^{\ast}_{B}(M,\, E)$ is a sub-complex of the de Rham 
complex $A^{\ast}(M,\,E)$ equipped with the differential $D$ associated to the flat 
connection.
We denote by $H^{\ast}(M,\,E) $ and $H^{\ast}_{B}(M,\,E) $ the cohomology of the de Rham complex  $A^{\ast}(M,\,E)$ and the cohomology of the sub-complex $A^{\ast}_{B}(M,\, E)$ respectively.
Suppose $E$ is equipped with a Hermitian metric $h$.
 Let
$$D=\,\nabla+\phi$$ 
be the canonical decomposition of the connection $D$ (see \eqref{code}).
Then, using the pairing $A^{r}(M,\,E)\times A^{2n+1-r}(M,\,E)\,\longrightarrow\, A^{2n+1}(M)$
associated with $h$, define the Hodge star operator $$\ast_{h}\,:\, A^{r}(M,\,E)
\,\longrightarrow\, A^{2n+1-r}(M,\,E)$$ as well as the formal adjoint operator
$D^{\ast}\,=\,-\ast_{h}D \ast_{h}$.

Assume the Hermitian structure $h$ to be basic (equivalently, $\phi(\xi)\,=\,0$  \cite[Proposition 4.1]{BK}).
By $D_{\xi}\,=\,\nabla_{\xi}$,  the unitary connection
$\nabla$ restricts to an operator
$$\nabla\,:\, A^{r}_{B}(M, \,E)\,\longrightarrow\,
A^{r+1}_{B}(M,\, E)\, .$$
Now, on $A^{p,q}_{B}(M,\, E)$, decompose $$\nabla\,=\,
\partial_{h,\xi}+\overline\partial_{h,\xi}$$ such that $\partial_{h,\xi}\,:\,
A^{p,q}_{B}(M,\, E) \,\longrightarrow\, A^{p+1,q}_{B}(M,\, E)$
and $\overline\partial_{h,\xi}\,:\,A^{p,q}_{B}(M,\, E)
\,\longrightarrow\, A^{p,q+1}_{B}(M,\, E)$.
Since $\nabla$ is a unitary connection, we have that $$\overline\partial_{\xi}h(s_{1},s_{2})
\,=\,h(\overline\partial_{h,\xi}s_{1}, s_{2})+h(s_{1}, \partial_{h,\xi}s_{2})$$ for
$s_1,\, s_2\,\in\, A^{0,0}_{B}(M,\, E)$.
We define the operator $$\star_{h, \xi}\,:\, A^{r}_{B}(M,\,E)\,\longrightarrow\,
A^{2n-r}_{B}(M,\,E)$$ and the formal adjoint operators
$$(\nabla)^{\ast}_{\xi}\,=\,-\star_{h, \xi}\nabla^{h}\star_{h,\xi}\,:\,
A^{r}_{B}(M,\,E)\,\longrightarrow\, A^{r-1}_{B}(M,\,E)\, ,$$
$$\partial_{h,\xi}^{\ast}\,=\,-\star_{h,\xi}\overline\partial_{h,\xi}\star_{h,\xi}\,:\,
A^{p,q}_{B}(M,\,E)\,\longrightarrow\, A^{p-1,q}_{B}(M,\,E)$$,
$$\overline\partial_{h,\xi}^{\ast}\,=\,-\star_{h,\xi}\partial_{h,\xi}\star_{h,\xi}\,:\,
A^{p,q}_{B}(M,\,E)\,\longrightarrow\,
A^{p,q-1}_{B}(M,\,E)\, ,$$ and
$\Lambda_{h} \,:=\,-\star_{h,\xi}(d\eta\wedge)\star_{h,\xi}$ in the same way as above.
We now have the K\"ahler identities
\begin{equation}\label{Kaid}
[\Lambda,\, \partial_{h,\xi}]\,=\,-\sqrt{-1}\overline\partial_{h,\xi}^{\ast} \ \
\text{ and }\ \ [\Lambda ,\,\overline\partial_{h,\xi}]\,=\, \sqrt{-1}\partial_{h,\xi}^{\ast}\,. 
\end{equation}

\begin{thm}{\rm (\cite[Theorem 4.2]{BK})}\label{thm1}
Let $(M,\, (T^{1,0},\, S,\, I),\, (\eta,\, \xi))$ be a compact Sasakian manifold and
$(E, \, D)$ a flat complex vector bundle over $M$ with a Hermitian metric $h$.
Then the following two conditions are equivalent:
\begin{itemize}
\item The Hermitian structure $h$ is harmonic, i.e., $(\nabla)^{\ast}\phi\,=\,0$.

\item The Hermitian structure $h$ is basic and for the 
decomposition $$\phi\,=\,\theta+\bar\theta$$ with 
$\theta\,\in\, A^{1,0}_{B}(M,\,{\rm End}(E))$ and $ 
\bar\theta\in A^{0,1}_{B}(M,\,{\rm End}(E))$, the equalities 
\[\overline\partial_{h,\xi} \overline\partial_{h,\xi} \,=\,0,\qquad 
[\theta,\,\theta]\,=\,0\qquad {\rm and} \qquad 
\overline\partial_{h,\xi} \theta\,=\,0\, \]
hold.
\end{itemize}
\end{thm}
For a semi-simple  flat complex vector bundle $(E, \, D)$ over a compact Sasakian manifold $M$,
we have a harmonic metric $h$  by Corlette's Theorem.
 On $A^{\ast}_{B}(M, \,E)$, we define the operators $D^{\prime}=\partial_{h,\xi}+\bar\theta:A^{\ast}_{B}(M, \,E)\to A^{\ast+1}_{B}(M, \,E)$ and $D^{\prime\prime}=\overline\partial_{h,\xi}+\theta:A^{\ast}_{B}(M, \,E)\to A^{\ast+1}_{B}(M, \,E)$.
By Theorem \ref{thm1}, we have $D^{\prime\prime}D^{\prime\prime}=0$.
Since $D$ is flat, $\nabla=\partial_{h,\xi}+\overline\partial_{h,\xi}$ is unitary and $\phi=\theta+\bar\theta$ is self-adjoint,
we have the equalities $D^{\prime}D^{\prime}=0$ and $D^{\prime }D^{\prime\prime}+D^{\prime\prime}D^{\prime}=0$.
 Define $D^{c}=\sqrt{-1}(D^{\prime\prime}-D^{\prime})$.
By using K\"ahler identities (\ref{Kaid}) and the similar equations for $\theta$ and $\bar\theta$,
we  have the  K\"ahler identities for the operators  $D^{\prime}$ and  $D^{\prime\prime}$
\begin{equation}\label{KaidT}
[\Lambda,\, D^{\prime}]\,=\,-\sqrt{-1}(D^{\prime\prime})^{\ast} \ \
\text{ and }\ \ [\Lambda ,\,D^{\prime\prime}]\,=\, (D^{\prime})^{\ast}\,
\end{equation}
as  \cite[Lemma 3.1]{SimC}.
By the similar arguments in \cite[Lemma 5.11]{DGMS}, we have the $DD^{c}$-Lemma:
\[{\rm ker}D\cap {\rm ker}D^{c}\cap {\rm  im}D={\rm ker}D\cap {\rm ker}D^{c}\cap {\rm  im}D^{c}={\rm  im}DD^{c}.
\]
Consider the sub-complex ${\rm ker}D^{c}\subset A^{\ast}_{B}(M, \,E)$ and the cohomology $H^{\ast}_{D^{c}}(A^{\ast}_{B}(M, \,E))$ of $A^{\ast}_{B}(M, \,E)$ for the differential $D^{c}$.
By the similar  arguments in \cite[Section 6]{DGMS}, the inclusion ${\rm ker}D^{c}\subset A^{\ast}_{B}(M, \,E)$ is a quasi-isomorphism, the differential on $H^{\ast}_{D^{c}}(A^{\ast}_{B}(M, \,E))$  induced by  $D$ is trivial and the quotient $q:{\rm ker}D^{c}\to H^{\ast}_{D^{c}}(A^{\ast}_{B}(M, \,E))$ is a quasi-isomorphism.
Hence, we have the sequence 
\begin{equation}\label{formal}
A^{\ast}_{B}(M, \,E) \leftarrow {\rm ker}D^{c} \to H^{\ast}_{B}(M, \,E)
\end{equation}
of quasi-isomorphisms.
By $d\eta\in A^{1,1}_{B}(M)$ and $\partial_{\xi}d\eta=\overline\partial_{\xi}d\eta=0$,
we have the sub-complexes 
$${\rm ker}D^{c} \oplus {\rm ker}D^{c}\wedge \eta\subset A^{\ast}_{B}(M, \,E)\oplus A^{\ast}_{B}(M, \,E)\wedge \eta\subset A^{\ast}(M, \,E)$$
and the cochain complex $H^{\ast}_{B}(M, \,E)\oplus H^{\ast}_{B}(M, \,E)\otimes \langle \eta \rangle$.
Define the decreasing  flirtation $F^{\ast}$ on $A^{\ast}_{B}(M, \,E)\oplus A^{\ast}_{B}(M, \,E)\wedge \eta$ (resp. $H^{\ast}_{B}(M, \,E)\oplus H^{\ast}_{B}(M, \,E)\otimes \langle \eta \rangle$) such that $F^{-1}=A^{\ast}_{B}(M, \,E)\oplus A^{\ast}_{B}(M, \,E)\wedge \eta$, $F^{0}= A^{\ast}_{B}(M, \,E)$ and $F^{1}=0$ and
define the one on $H^{\ast}_{B}(M, \,E)\oplus H^{\ast}_{B}(M, \,E)\otimes \langle \eta \rangle$ in the same manner.
By the standard argument on the spectral sequences of these filtrations,
we have  the sequence 
\begin{equation}\label{etaformal}
A^{\ast}_{B}(M, \,E)\oplus A^{\ast}_{B}(M, \,E)\wedge \eta \leftarrow {\rm ker}D^{c} \oplus {\rm ker}D^{c}\wedge \eta \to H^{\ast}_{B}(M, \,E)\oplus H^{\ast}_{B}(M, \,E)\otimes \langle \eta \rangle
\end{equation}
of quasi-isomorphisms.

\begin{prop}
The inclusion $A^{\ast}_{B}(M, \,E)\oplus A^{\ast}_{B}(M, \,E)\wedge \eta\subset A^{\ast}(M, \,E)$ is a quasi-isomorphism.

\end{prop}
\begin{proof}

Consider the covariant derivative $D_{\xi}$ on $A^{\ast}(M, \,E)$ at the Reeb vector field $\xi$.
Then, a form $\omega\in A^{\ast}(M, \,E)$ is basic if and only if 
$D_{\xi}\omega=i_{\xi}\omega=0$.
For any $\omega\in {\rm ker} D_{\xi}$, we have the decomposition $$\omega=(\omega-i_{\xi}\omega\wedge \eta)+i_{\xi}\omega\wedge \eta\in A^{\ast}_{B}(M, \,E)\oplus A^{\ast}_{B}(M, \,E)\wedge \eta$$ and hence 
we can say  ${\rm ker} D_{\xi}=A^{\ast}_{B}(M, \,E)\oplus A^{\ast}_{B}(M, \,E)\wedge \eta$.
Since we have $D_{\xi}=\nabla_{\xi}$ for the unitary connection $\nabla$ and $\xi$ is Killing, we have  $(D_{\xi})^{\ast}=- D_{\xi}$ and hence $D_{\xi}$ and the  Laplacian operator $\Delta_{D}=DD^{\ast}+D^{\ast}D$ commute.
Hence, if $\omega\in A^{\ast}(M, \,E)$ is harmonic, then $D_{\xi}\omega$ is also harmonic.
Since $D_{\xi}$ induces a trivial map on the cohomology $H^{\ast}(M, \,E)$ of $A^{\ast}(M, \,E)$, we have $D_{\xi}\omega=0$.
Thus, the inclusion  $ {\rm ker} D_{\xi}=A^{\ast}_{B}(M, \,E)\oplus A^{\ast}_{B}(M, \,E)\wedge \eta\subset A^{\ast}(M, \,E)$
induces a surjection on cohomology.

Consider the operator $H_{D}$ on $A^{\ast}(M, \,E) $ defined as the projection to harmonic forms.
By the above argument, we have $D_{\xi}H_{D}=0$. 
Since $H_{D}$ is self-adjoint, we have $H_{D}D_{\xi}=0$.
Consider the Green operator $G_{D}$ for  $\Delta_{D}$.
For  $\omega\in  {\rm ker} D_{\xi}$, we have $\Delta_{D}D_{\xi}G\omega= D_{\xi}\Delta_{D}G\omega= D_{\xi}(\omega-H_{D}\omega)=0$ and hence $D_{\xi}G\omega=H_{D}D_{\xi}G\omega=0$.
If $D\omega=0$, then we have $\omega=H_{D}(\omega)+DD^{\ast}G\omega$ with $G\omega\in {\rm ker} D_{\xi}$.
We notice that $D^{\ast}$ preserves $ {\rm ker} D_{\xi}$ since $D_{\xi}$ and $D^{\ast}$ commute.
If $[\omega]=0$ in $H^{\ast}(M, \,E)$, then $H_{D}(\omega)=0$ and so we have  $\omega=DD^{\ast}G\omega$ for $D^{\ast}G\omega \in {\rm ker} D_{\xi}$.
Thus, the inclusion  $ {\rm ker} D_{\xi}=A^{\ast}_{B}(M, \,E)\oplus A^{\ast}_{B}(M, \,E)\wedge \eta\subset A^{\ast}(M, \,E)$
induces an injection on cohomology.
Hence the proposition follows.

\end{proof}

\begin{cor}\label{alfo}
We have  the sequence 
\begin{equation}\label{etaformal}
A^{\ast}(M, \,E)\leftarrow {\rm ker}D^{c} \oplus {\rm ker}D^{c}\wedge \eta \to H^{\ast}_{B}(M, \,E)\oplus H^{\ast}_{B}(M, \,E)\otimes \langle \eta \rangle
\end{equation}
of quasi-isomorphisms.
\end{cor}

\begin{rem}
If  $E$ is trivial, then this statement is proved in \cite{Ti}.
\end{rem}

\section{Deformations of representations of  the fundamental groups of Sasakian manifolds}

We review the work of Goldman-Millson (\cite{GM}).
Let $M$ be a compact manifold with a point $x$.
We consider the variety  ${\mathcal R}(\pi_{1}(M,x), GL_{m}(\C))$ of representations of the fundamental group $\pi_{1}(M,x)$.
This is the set ${\rm Hom}(\pi_{1}(M,x),GL_{m}(\C)) $ with the natural structure of an algebraic variety induced by the group structure of $\pi_{1}(M,x)$ and the algebraic group $GL_{m}(\C)$.
For $\rho\in  {\mathcal R}(\pi_{1}(M,x), GL_{m}(\C))$, we define the flat vector bundle $(E=\pi_{1}(M,x)\backslash(\tilde{M}\times \C^{m}), D)$ where $\tilde{M}$ is the universal covering associated with the base point $x$.
Consider the differential graded Lie algebra $A^{\ast}(M, \, {\rm End}(E))$ and the augmentation map $\epsilon_{x}:A^{0}(M, \, {\rm End}(E))\to {\rm End}(E_{x})$.
Then the analytic germ of ${\mathcal R}(\pi_{1}(M,x), GL_{m}(\C))$ at $\rho$  pro-represents the "deformation  functor" of  the differential graded Lie algebra $A^{\ast}(M, \, {\rm End}(E))_{x}={\rm ker} \epsilon_{x}$ (Theorem \cite[Theorem 6.8]{GM}).
If a finite dimensional  ${\rm End}(E_{x})$-augmented differential graded Lie algebra $L^{\ast}$ with an injective augmentation $\epsilon:L^{0}\to  {\rm End}(E_{x})$  is quasi-isomorphic to the   ${\rm End}(E_{x})$-augmented differential graded Lie algebra $A^{\ast}(M, \, {\rm End}(E))$, then the deformation functor of $A^{\ast}(M, \, {\rm End}(E))_{x}$ is pro-represented by the analytic germ of 
$$\left\{\omega\in L^{1} \left\vert d\omega +\frac{1}{2}[\omega,\omega]=0\right\}\right.\times {\rm End}(E_{x})/\epsilon(L^{0})$$ at the origin(see \cite[Section 3]{GM}).

\begin{thm}
Let $M$ be a $(2n+1)$-dimensional compact Sasakian manifold with a point $x$.
If $n\ge 2$, then for any semi-simple representation $\rho :\pi_{1}(M,x)\to GL_{m}(\C)$, the analytic germ of ${\mathcal R}(\pi_{1}(M,x), GL_{m}(\C))$ at $\rho$  is quadratic.
\end{thm}
\begin{proof}
By the above arguments and Corollary \ref{alfo}, it is sufficient to prove that the analytic germ of 
\[\left\{\omega\in H^{1}_{B}(M, \,{\rm End}(E))\oplus H^{0}_{B}(M, \,{\rm End}(E))\otimes \langle \eta \rangle \left\vert d\omega +\frac{1}{2}[\omega,\omega]=0\right\}\right.\]
 at the origin is quadratic.
For $\omega=\alpha+\beta\otimes \eta \in H^{1}_{B}(M, \,{\rm End}(E))\oplus H^{0}_{B}(M, \,{\rm End}(E))\otimes \langle \eta \rangle$, the equality  $d\omega +\frac{1}{2}[\omega,\omega]=0$ holds if and only if $\beta\wedge d\eta+\frac{1}{2}[\alpha,\alpha]=0$ and 
$[\alpha,\beta]\otimes \eta=0$.
On the other hand, $\beta\wedge d\eta+\frac{1}{2}[\alpha,\alpha]=0$ implies
\[[\alpha,\beta]\wedge d\eta=[\alpha,\beta \wedge d\eta]=-\frac{1}{2}[\alpha,[\alpha,\alpha]]=0
\]
by the graded Jacobi identity.
By the K\"ahler identities (\ref{KaidT}), if $n\ge 2$,  we can say that the map $H^{1}_{B}(M, \,{\rm End}(E))\ni a\mapsto a\wedge d\eta\in H^{3}_{B}(M, \,{\rm End}(E))$ is injective as the usual Lefschetz decomposition.
Hence, if $n\ge 2$, $\beta\wedge d\eta+\frac{1}{2}[\alpha,\alpha]=0$ implies $[\alpha,\beta]=0$ and 
this means that the equality  $d\omega +\frac{1}{2}[\omega,\omega]=0$ is equivalent to the quadratic  equation $\beta\wedge d\eta+\frac{1}{2}[\alpha,\alpha]=0$.
Hence the theorem follows.
\end{proof}

\begin{rem}
If $\rho$ is trivial, the statement follows from  \cite[Corollary 6.10]{PS} (see also \cite{Kas}) 

\end{rem}

\section{Vanishing cup products}

\begin{thm}
Let $M$ be a $(2n+1)$-dimensional compact Sasakian manifold and $E$ and $E^{\prime}$ be semi-simple  flat complex vector bundles.
For $s, t<n$ with $s+t> n$,  the cup product
\[H^{s}(M, \,E)\otimes H^{t}(M, \,E^{\prime})\to H^{s+t}(M, \,E\otimes E^{\prime})
\]
vanishes.
\end{thm}
\begin{proof}
Considering  the quasi-isomorphisms as in Corollary \ref{alfo} for  $E$, $E^{\prime}$ and $E\otimes E^{\prime}$,
we can say that 
the cup product $H^{\ast}(M, \,E)\otimes H^{\ast}(M, \,E^{\prime})\to H^{\ast}(M, \,E\otimes E^{\prime})$ 
is induced by the product 
\begin{multline*}
 \left(H^{\ast}_{B}(M, \,E)\oplus H^{\ast}_{B}(M, \,E)\otimes \langle \eta \rangle\right) \bigotimes \left(H^{\ast}_{B}(M, \,E^{\prime})\oplus H^{\ast}_{B}(M, \,E^{\prime})\otimes \langle \eta \rangle\right) 
\\
 \to H^{\ast}_{B}(M, \,E\otimes E^{\prime})\oplus H^{\ast}_{B}(M, \,E\otimes E^{\prime})\otimes \langle \eta \rangle .
\end{multline*}
By the K\"ahler identities (\ref{KaidT}),  we can say that the map $H^{r}_{B}(M, \,E)\ni a\mapsto a\wedge d\eta\in H^{r+2}_{B}(M, \,{\rm End}(E))$ is injective for $r\le n-1$ and surjective for $r\ge n-1$
 as the usual Lefschetz decomposition.
 Thus, any $r$-th cohomology class of the complex $H^{\ast}_{B}(M, \,E)\oplus H^{\ast}_{B}(M, \,E)\otimes \langle \eta\rangle$ has a representative in  $H^{r}_{B}(M, \,E)$ for $r\le n$.
By the injectivity,  for $s, t<n$,
any cup product in 
\begin{multline*}
H^{s} \left(H^{\ast}_{B}(M, \,E)\oplus H^{\ast}_{B}(M, \,E)\otimes \langle \eta \rangle\right)\bigotimes H^{t}\left(H^{\ast}_{B}(M, \,E^{\prime})\oplus H^{\ast}_{B}(M, \,E^{\prime})\otimes \langle \eta \rangle\right) \\
\to H^{s+t}\left(H^{\ast}_{B}(M, \,E\otimes E^{\prime})\oplus H^{\ast}_{B}(M, \,E\otimes E^{\prime})\otimes \langle \eta \rangle\right)
\end{multline*}
has a representative in $H^{s+t}_{B}(M, \,E\otimes E^{\prime})$.
By the surjectivity,  if $s+t>n$, any element in $H^{s+t}_{B}(M, \,E\otimes E^{\prime})$ is exact.
Hence the theorem follows.

\end{proof}

\begin{rem}
If  $E$ and $E^{\prime}$ are trivial, the statement follows from  \cite{Bu}.

\end{rem}

\end{document}